\documentclass[12pt]{article}
\pdfoutput=1
\usepackage{amsmath,amssymb,amsthm}
\usepackage{graphicx}
\usepackage{enumerate}
\usepackage{fullpage}

\usepackage{array}

\usepackage{pifont}
\newcommand{\cmark}{\ding{51}}%
\newcommand{\xmark}{\ding{55}}%
\usepackage{mathtools}
\usepackage{mathrsfs}	

\usepackage[font=footnotesize, labelfont=bf, margin=0.5cm]{caption}

\theoremstyle{plain}
\newtheorem{thm}{Theorem}
\newtheorem{conj}{Conjecture}
\newtheorem{quest}[conj]{Question}

\usepackage[sc]{mathpazo}
\newcommand{\olx}{\overline{x}}
\newcommand{\oly}{\overline{y}}

\usepackage[colorlinks=true,citecolor=blue,linkcolor=red,urlcolor=blue]{hyperref}

\title{Exact solution of some quarter plane walks with interacting boundaries}
\author{N.\ R.\ Beaton\thanks{\href{mailto:nrbeaton@unimelb.edu.au}{nrbeaton@unimelb.edu.au}} \hspace{1cm}
A.\ L.\ Owczarek\thanks{\href{mailto:owczarek@unimelb.edu.au}{owczarek@unimelb.edu.au}}\\
\small School of Mathematics and Statistics\\
\small The University of Melbourne, Victoria 3010, Australia
\and
A.\ Rechnitzer\thanks{\href{mailto:andrewr@math.ubc.ca}{andrewr@math.ubc.ca}}\\
\small Department of Mathematics\\
\small The University of British Columbia, Vancouver V6T 1Z2, British Columbia, Canada
}

\begin{document}

\maketitle
\begin{abstract}
  The set of random walks with different step sets (of short steps) in the quarter plane has provided a rich set of models that have profoundly different integrability properties. In particular, 23 of the 79 effectively different models can be shown to have generating functions that are algebraic or differentiably finite. Here we investigate how this integrability may change in those 23 models where in addition to length one also counts the number of sites of the walk touching either the horizontal and/or vertical boundaries of the quarter plane. This is equivalent to introducing interactions with those boundaries in a statistical mechanical context. We are able to solve for the generating function in a number of cases. For example, when counting the total number of boundary sites without differentiating whether they are horizontal or vertical, we can solve the generating function of a generalised Kreweras model. However, in many instances we are not able to solve as the kernel methodology seems to break down when including counts with the boundaries.
\end{abstract}

\section{Introduction}

The study of lattice random walks has a long history \cite{barber1970, gessel1998, polya1921, spitzer2013}.
It is relatively straightforward to analyse random walks with different step sets in two and three dimensions where the walks are unrestricted.
Once boundaries are introduced to restrict the region the problem can become far more difficult and in some cases no solution has been found for what might seem simple regions.
One restriction that has received much attention is to consider walks in the first quadrant of the plane with some subset of the 8 ``short'' steps $\{\text{N, E, S, W, NE, NW, SE, SW}\}$
--- see, amongst many others \cite{bostan2009automatic, bousquet2002counting, bousquet-melou_walks_2010, dreyfus2018, fayolle1999random, gessel1992, kauers2009, raschel2012}.
As usual from a combinatorial viewpoint the length generating function of all walk configurations is the mathematical object of interest.

In \cite{bostan2009automatic, bousquet-melou_walks_2010, mishna2009} the 256 possible walk models are distilled into different equivalence classes.
Of these classes, 79 (representing 138 distinct step sets) were identified as being non-trivial.
Of these 79, only 23 were determined to have algebraic or D-finite \cite{lipshitz1989, stanley2001} generating functions.
In this paper we consider those 23  cases and extend the model by counting the number of times the random walk visits the horizontal and vertical boundaries.
In the statistical physics and probability literatures this type of extension is described as adding an \emph{interaction} on the boundary.
Interaction problems are important as they often lead to so-called phase transitions \cite{buks2015} where the behaviour of the system changes markedly as the variable corresponding to surface visits is varied \cite{tabbara2013, tabbara2016}.
Here our interest is in how counting the visits to the boundary changes the solvability of the problem and the analytic character of the generating function.
In other words, how sensitive is that analytic character to the addition of boundary interactions?

\subsection{The model}
Consider a random walk on $\mathbb{Z}^2$ with steps $\mathcal{S}$ taken from the set $\{-1,0,1\}^2\setminus\{(0,0)\}$.
The walks are restricted to lie in the non-negative quadrant --- such random walks are known as \emph{quarter plane walks}.
A now classical problem is to enumerate the number of such random walks of length $n$ starting at the origin and then ending back at the origin, or ending anywhere in the quarter plane.
In this paper we consider only those walks that start and end at the origin.

For a specific step set, let $q_n$ be the number of walks of length $n$ that start and end at the origin. We associate a generating function
\begin{align}
  G(t) &= \sum_{n \geq 0} q_n t^n.
\end{align}
In order to construct this generating function we need to form associated generating functions which count walks according to the coordinate of their end-point.
To this end define
\begin{align}
  Q(t;x,y)  &= \sum_{n \geq 0} t^n \sum_{k,\ell \geq 0} q_{n,k,\ell} x^k y^\ell \;,
\end{align}
where $q_{n,k,\ell}$ is the number of walks of length $n$ ending at $(k,\ell)$.
The nature of the generating functions $Q(t;0,0)$ and $Q(t;1,1)$ has been the focus of much work in this area.

\begin{figure}
  \centering
	\includegraphics[width=0.35\textwidth]{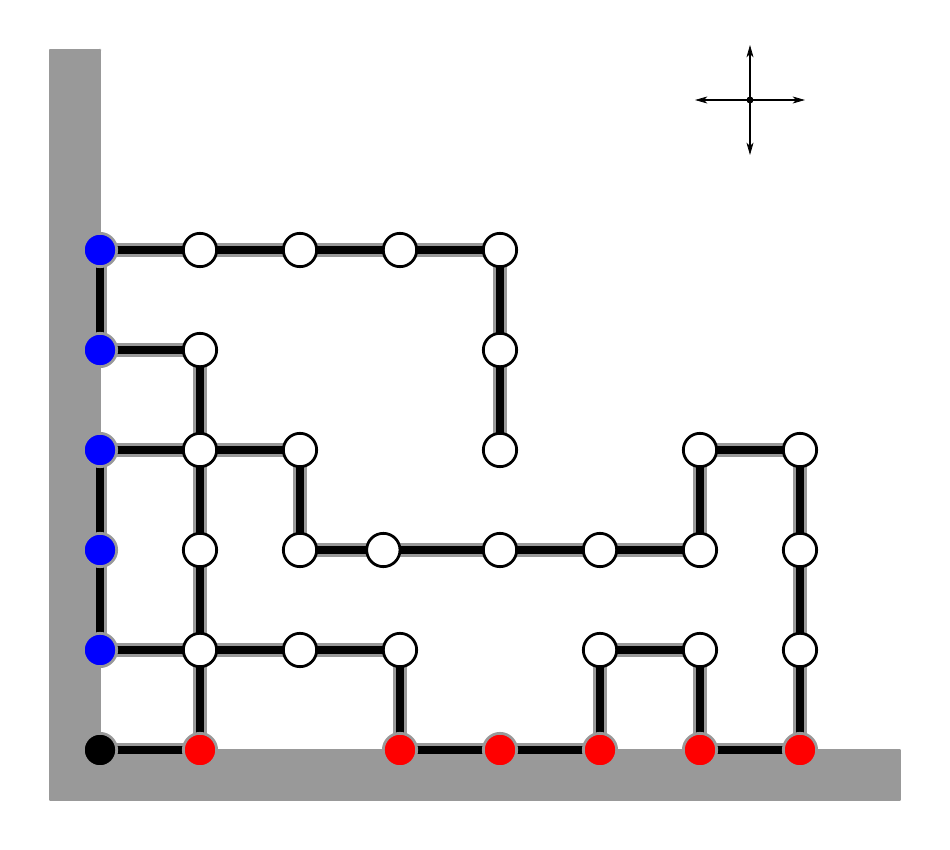}
	\hspace{2cm}
  \includegraphics[width=0.35\textwidth]{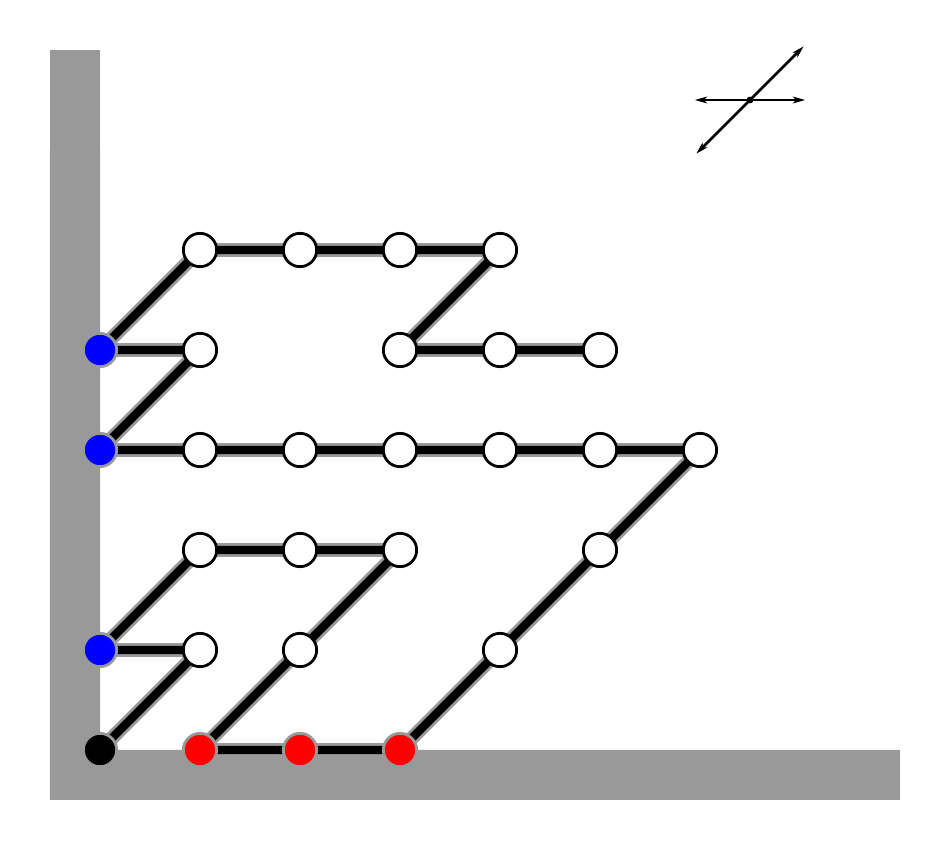}
  \caption{Two examples of random walks in the quarter plane with visits to either boundary wall indicated.
  The walk on the left may take N, S, E and W steps (see Model 1 below), while the walk on the right may take E, W, NE, and SW steps (see Model 23 below).  The main focus of this paper is similar walks constrained to start and end at the origin.}
  \label{fig qp eg}
\end{figure}

We extend these enumeration problems by also counting the number of times walks visit either the horizontal or vertical boundaries.
Let $q_{n,k,\ell,h,v}$ be the number of walks of length $n$ that end at $(k,\ell)$ which visit the horizontal boundary (ie.~the line $y=0$) $h$ times, and visit the vertical boundary (ie.~the line $x=0$) $v$ times.
The associated generating function is
\begin{align}
  Q(t;x,y;a,b) \equiv Q(x,y) &= \sum_{n \geq 0} t^n \sum_{k,\ell,h,v \geq 0} q_{n,k,\ell,h,v} x^k y^\ell a^h b^v \equiv \sum_{n\geq 0} t^n Q_n(x,y).
\end{align}
A visit to the origin gains a weight $ab$ as the origin is considered as part of both the horizontal and vertical boundaries. We will, where it is clear, suppress the variables $t,a$ and $b$. We then redefine
\begin{align}
  G(t;a,b) &= Q(t;0,0;a,b).
\end{align}
We examine four subcases: $G(t;a,1), G(t;1,b), G(t;a,a)$ and $G(t;a,b)$.
These count (respectively) visits only to the horizontal boundary, visits only to the vertical boundary, visits to either boundary (not distinguishing between horizontal and vertical), and all visits to the boundaries (distinguishing between horizontal and vertical).

The nature of these generating functions for different values of $a,b$ is the focus of this paper.

\subsection{Main results}
As noted above we focus on the 23 classes of step sets where the generating function for the non-interacting models is D-finite or algebraic: these are described in the tables below.
We recall that a formal series in $t$ is D-finite \cite{lipshitz1989, stanley2001} if it satisfies a linear differential equation with polynomial coefficients.
Our main results can be summarised as saying that the solubility of 22 of the 23 classes changes with the introduction of boundary weights $a$ and $b$.
We follow the numbering from \cite{bousquet-melou_walks_2010} and then separate the different walk classes according to the size of the symmetry group associated with their step set.

Each of Tables~\ref{tab 1to4}--\ref{tab 2223} contains information on the nature of the solution for various values of the boundary weights.
\begin{itemize}
  \item \textbf{Alg} --- the solution is algebraic (and not rational).
  If a checkmark \cmark\ is present, then this has been proved.
  If a cross \xmark\ is present, then this has been guessed on the basis of analysis of series data using the powerful Ore Algebra package \cite{kauers2015ore} in the Sage computer algebra system \cite{sagemath}.
  \item \textbf{DF} --- the solution is D-finite. Again this has been proved where a \cmark\ is present and guessed where there is an \xmark.
  In all the proven cases the solution is not algebraic. In the guessed cases we know that the solution cannot be algebraic if the solution at $(a,b)=1$ is not algebraic.
  This leaves Models 21 and 23; we have guessed D-finite solutions for some parameter ranges, but we have not been able to find algebraic solutions.
  \item \textbf{DAlg} --- the solution is D-algebraic. That is, the generating function $G(t;a,b)$ and its derivatives (with respect to $t$) satisfy a non-trivial polynomial.
  We use a \cmark\ to denote that the solution can be written as the ratio of two (known) D-finite functions and hence is D-algebraic.
  We have not been able to guess D-algebraic solutions to any of the models.
  \item \textbf{?} --- we have not been able to solve or guess any algebraic, D-finite or D-algebraic solution.
\end{itemize}
In addition we give information about the so called \emph{group of the walk} \cite{fayolle1999random}.
This is the group of symmetries satisfied by the kernel --- see equation~\eqref{eqn kern defn}.
We state the order of the group and the two generating involutions.
We have used the notation $\olx \equiv x^{-1}$ and  $\oly \equiv y^{-1}$.

\subsubsection{Vertically and horizontally symmetric --- $D_2$ groups}
The 4 walk models in Table~\ref{tab 1to4} possess both vertical and horizontal symmetry; the group of symmetries associated with their step set is equivalent to $D_2$.
In all cases we are able to find the generating function for all $(a,b)$ values.

\begin{thm}\label{thm 1234}
  For general values of $a$ and $b$, the generating function $G(t;a,b)$ for Models $1,3$ and $4$ can be written as the ratio of two D-finite functions.
  This simplifies when either $a$ or $b$ is one, and $G(t;a,1)$ and $G(t;1,b)$ are D-finite. On the other hand, for Model $2$, $G(t;a,b)$ is D-finite for all $a,b$.
\end{thm}
\begin{proof}
  See Sections~\ref{sec soln 134} and~\ref{sec soln 2}.
\end{proof}

\begin{table}[h!]
\begin{center}
\begin{tabular}{|c| m{11mm} |c|c|c|c|c|c|}
  \hline
  & \multicolumn{1}{c|}{Model} & (1,1) & (a,a) & (a,1) & (1,b) & (a,b) & Group \\
  \hline
  1 & \vspace{4pt}\includegraphics[height=1cm]{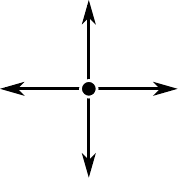} &
  DF  \cmark & DAlg \cmark & DF \cmark & DF \cmark & DAlg \cmark & 4: $(\olx,y), (x,\oly)$\\
  \hline \hline
  2 & \vspace{4pt}\includegraphics[height=1cm]{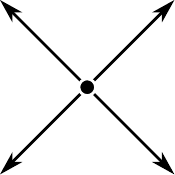} &
  DF  \cmark & DF \cmark & DF \cmark & DF \cmark & DF \cmark & 4: $(\olx,y), (x,\oly)$\\
  \hline \hline
  3 & \vspace{4pt}\includegraphics[height=1cm]{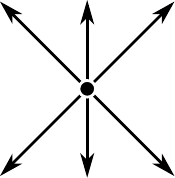} &
  DF  \cmark & DAlg \cmark & DF \cmark & DF \cmark & DAlg \cmark & 4: $(\olx,y), (x,\oly)$\\
  \hline
  4 & \vspace{4pt}\includegraphics[height=1cm]{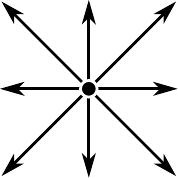} &
  DF  \cmark & DAlg \cmark & DF \cmark & DF \cmark & DAlg \cmark & 4: $(\olx,y), (x,\oly)$\\
  \hline
\end{tabular}
\caption{The step sets associated with these walk models are both vertically and horizontally symmetric. The results are summarised in Theorem~\ref{thm 1234} and the solutions are given in Sections~\ref{sec soln 134} and~\ref{sec soln 2}.}
\label{tab 1to4}
\end{center}
\end{table}

\subsubsection{Horizontally symmetric --- $D_2$ groups}
The 12 walk models in Tables~\ref{tab 5to10} and~\ref{tab 11to16} possess horizontal symmetry --- the step set is symmetric over the line $x=0$.
The models in Table~\ref{tab 5to10} have a net positive vertical drift, while those in Table~\ref{tab 11to16} have a net negative vertical drift.
We were only able to find solutions when $b=1$ --- preserving the symmetry of the model.
Indeed the solution method is essentially the same as for $a=b=1$ and the generating function remains D-finite.

When $b\neq 1$ we were not able to find or guess a solution for any of the models. We discuss this further in Section~\ref{sec soln 5to16}.
\begin{thm}\label{thm 5to16}
  The generating function $G(t;a,1)$ is D-finite for Models 5 to 16.
\end{thm}
\begin{proof}
  See Section~\ref{sec soln 5to16}.
\end{proof}

\begin{table}
\begin{center}
\begin{tabular}{|c|m{11mm}|c|c|c|c|c|c|c|c|}
  \hline
  & \multicolumn{1}{c|}{Model} & (1,1) & (a,a) & (a,1) & (1,b) & (a,b) & Group\\
  \hline
  5 & \vspace{4pt}\includegraphics[height=1cm]{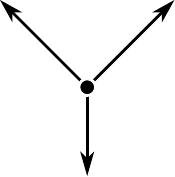} &
  DF  \cmark& ? & DF \cmark & ? & ? & 4: $(\olx,y), \left(x,\frac{1}{x+\olx}\oly\right)$\\
  \hline
  6 & \vspace{4pt}\includegraphics[height=1cm]{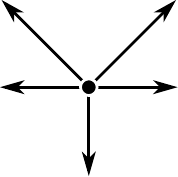} &
  DF  \cmark& ? & DF \cmark & ? & ? & 4: $(\olx,y), \left(x,\frac{1}{x+\olx}\oly\right)$\\
  \hline
  7 & \vspace{4pt}\includegraphics[height=1cm]{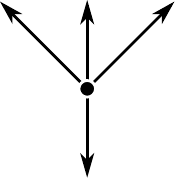} &
  DF  \cmark& ? & DF \cmark & ? & ? & 4: $(\olx,y), \left(x,\frac{1}{x+1+\olx}\oly\right)$\\
  \hline
  8 & \vspace{4pt}\includegraphics[height=1cm]{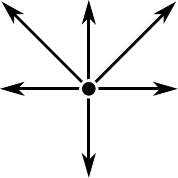} &
  DF  \cmark& ? & DF \cmark & ? & ? & 4: $(\olx,y), \left(x,\frac{1}{x+1+\olx}\oly\right)$\\
  \hline
  9 & \vspace{4pt}\includegraphics[height=1cm]{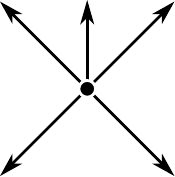} &
  DF  \cmark& ? & DF \cmark & ? & ? & 4: $(\olx,y), \left(x,\frac{x+\olx}{x+1+\olx}\oly\right)$\\
  \hline
  10 & \vspace{4pt}\includegraphics[height=1cm]{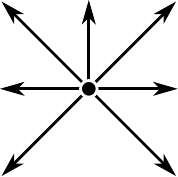} &
  DF  \cmark& ? & DF \cmark & ? & ? & 4: $(\olx,y), \left(x,\frac{x+\olx}{x+1+\olx}\oly\right)$\\
  \hline
\end{tabular}
\end{center}
\caption{Step sets that are symmetric over the line $x=0$ with a net positive vertical drift. We have only found solutions when $b=1$; the results are summarised in Theorem~\ref{thm 5to16} and the solutions are given in Section~\ref{sec soln 5to16}.}
\label{tab 5to10}
\end{table}

\begin{table}
\begin{center}
\begin{tabular}{|c|m{11mm}|c|c|c|c|c|c|c|c|}
  \hline
  & \multicolumn{1}{c|}{Model} & (1,1) & (a,a) & (a,1) & (1,b) & (a,b) & Group\\
  \hline
  11 & \vspace{4pt}\includegraphics[height=1cm]{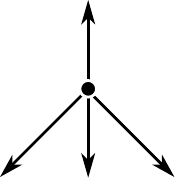} &
  DF  \cmark& ? & DF \cmark & ? & ? & 4: $(\olx,y), \left(x,(x+1+\olx)\oly\right)$\\
  \hline
  12 & \vspace{4pt}\includegraphics[height=1cm]{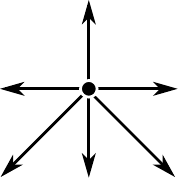} &
  DF  \cmark& ? & DF \cmark & ? & ? & 4: $(\olx,y), \left(x,(x+1+\olx)\oly\right)$\\
  \hline
  13 & \vspace{4pt}\includegraphics[height=1cm]{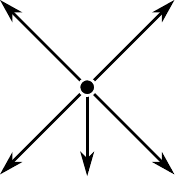} &
  DF  \cmark& ? & DF \cmark & ? & ? & 4: $(\olx,y), \left(x,\frac{x+1+\olx}{x+\olx}\oly\right)$\\
  \hline
  14 & \vspace{4pt}\includegraphics[height=1cm]{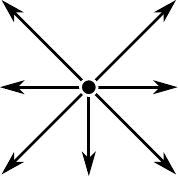} &
  DF  \cmark& ? & DF \cmark & ? & ? & 4: $(\olx,y), \left(x,\frac{x+1+\olx}{x+\olx}\oly\right)$\\
  \hline
  15 & \vspace{4pt}\includegraphics[height=1cm]{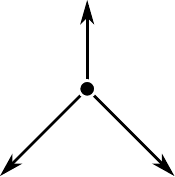} &
  DF  \cmark& ? & DF \cmark & ? & ? & 4: $(\olx,y), \left(x,(x+\olx)\oly\right)$\\
  \hline
  16 & \vspace{4pt}\includegraphics[height=1cm]{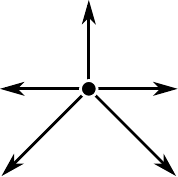} &
  DF  \cmark& ? & DF \cmark & ? & ? & 4: $(\olx,y), \left(x,(x+\olx)\oly\right)$\\
  \hline
\end{tabular}
\end{center}
\caption{Step sets that are symmetric over the line $x=0$ with a net negative vertical drift. We have only found solutions when $b=1$; the results are summarised in Theorem~\ref{thm 5to16} and the solutions are given in Section~\ref{sec soln 5to16}.}
\label{tab 11to16}
\end{table}

\subsubsection{$D_3$ groups}
Models 17 to 21 have diagonal symmetries rather than vertical or horizontal ones, and the corresponding group is isomorphic to $D_3$.
The solutions to these models are varied.
\begin{thm}\label{thm 1718}
  For general values of $a$ and $b$, the generating function $G(t;a,b)$ for Model $17$ can be written as the ratio of two D-finite functions.
  This simplifies when either $a$ or $b$ is one, and $G(t;a,1)$ and $G(t;1,b)$ are D-finite.
  On the other hand, for Model $18$, $G(t;a,a)$ can be written as the ratio of two D-finite functions.
\end{thm}
\begin{proof}
  See Section~\ref{sec soln 17} for a solution of Model 17, and Section~\ref{sec soln 18} for a solution of Model 18.
\end{proof}
Model~18 is very similar to the model considered in \cite{tabbara2016}; the main difference being the presence of a weighted neutral step.
The solution we give in Section~\ref{sec soln 18} is essentially unchanged from that in \cite{tabbara2016}.

Notice that for Model~18, we have been able to guess D-finite solutions when either $a$ or $b$ is one, but have not been able to prove these rigorously.
Additionally, for general $a,b$ we have not been able to find or guess any solution.
We note that Model 17 has a single diagonal symmetry along the line $y=-x$ which is not that of the quarter plane itself which is symmetric on the line $y=x$.
On the other hand Model 18 has two diagonal symmetries in the line $y=-x$ and the quarter plane symmetry of $y=x$.
How this observation might be related to our difficulty in proving anything about Model 18 when $a\neq b$ which breaks the $y=x$ symmetry is an intriguing question.

\begin{table}[h!]
\begin{center}
\begin{tabular}{|c|m{11mm}|c|c|c|c|c|c|c|c|}
  \hline
  & \multicolumn{1}{c|}{Model} & (1,1) & (a,a) & (a,1) & (1,b) & (a,b) & Group\\
  \hline
  17 & \vspace{4pt}\includegraphics[height=1cm]{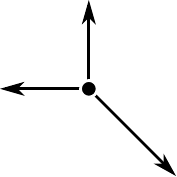} &
  DF  \cmark& DAlg \cmark & DF \cmark & DF \cmark & DAlg \cmark & 6: $(\olx y,y), (x,x\oly)$\\
  \hline
  18 & \vspace{4pt}\includegraphics[height=1cm]{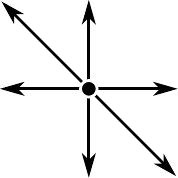} &
  DF  \cmark& DAlg \cmark & DF \xmark & DF \xmark & ? & 6: $(\olx y,y), (x, x\oly)$\\
  \hline
\end{tabular}
\end{center}
\caption{Step sets associated with the group generated by the involutions given in the final column that generate a group isomorphic to $D_3$.}
\label{tab 1718}
\end{table}

Models 19, 20 and 21 are the Kreweras, reverse Kreweras and double Kreweras models
\cite{bousquet2005, bousquet-melou_walks_2010, kreweras1965, mishna2009}, and have algebraic generating functions when $a=b=1$.
The only other model that does so is Model 23, Gessel walks \cite{bernadicounting, gessel1986, kauers2009}. We note Models 19 and 20 have a single diagonal symmetry $y=x$; the same as the quarter plane. Once again we can solve these models when that symmetry is preserved by setting $a=b$ in the boundary weights.
\begin{thm}\label{thm 192021}
  For Models 19 and 20, the generating function $G(t;a,a)$ is algebraic.
\end{thm}
\begin{proof}
Since Models 19 and 20 are reverses of each other, any walk in Model 19 that starts and ends at the origin, is also a walk in Model 20 simply traversed backwards.
Hence $G(t;a,b)$ is identical in the two models. See Section~\ref{sec soln 19} for a derivation of $G(t;a,a)$.
\end{proof}
Unfortunately, we have not been been able to solve either Models 19 or 20 when the $x=y$ symmetry is broken.
Despite this we have observed (again guessing via the Ore Algebra package) that $G(t;a,b)$ is algebraic for general $(a,b)$; $G(t;a,1)$ and $G(t;1,b)$ appear to satisfy degree 6 equations, while $G(t;a,b)$ appears to satisfy a degree 12 equation.

For Model 21, which has the two diagonal symmetries, we have guessed D-finite solutions when either $a$ or $b$ is one, but we have not found algebraic solutions. We have been unable to find any solution when $a=b$ or at general $a,b$.

\begin{table}[h!]
\begin{center}
\begin{tabular}{|c|m{11mm}|c|c|c|c|c|c|c|c|}
  \hline
  & \multicolumn{1}{c|}{Model} & (1,1) & (a,a) & (a,1) & (1,b) & (a,b) & Group\\
  \hline
  19 & \vspace{4pt}\includegraphics[height=1cm]{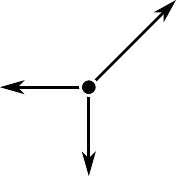} &
  Alg  \cmark& Alg \cmark & Alg \xmark & Alg \xmark & Alg \xmark & 6: $(\olx \oly,y), (x,\olx \oly)$\\
  \hline
  20 & \vspace{4pt}\includegraphics[height=1cm]{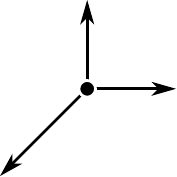} &
  Alg \cmark& Alg \cmark & Alg \xmark & Alg \xmark & Alg \xmark & 6: $(\olx \oly,y), (x,\olx \oly)$\\
  \hline
  21 & \vspace{4pt}\includegraphics[height=1cm]{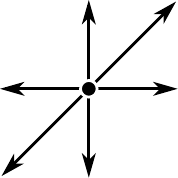} &
  Alg  \cmark& ? & DF \xmark & DF \xmark & ? & 6: $(\olx \oly,y), (x,\olx \oly)$\\
  \hline
\end{tabular}
\end{center}
\caption{Step sets associated with the group generated by the involutions given in the final column that generate a group isomorphic to $D_3$.}
\label{tab 19to21}
\end{table}

\subsubsection{$D_4$ groups}
The final two step sets are neither symmetric under a single horizontal, vertical or either diagonal reflection. However, they are symmetric when double reflections are considered and the group of the walk is isomorphic to $D_4$. Nevertheless the integrability properties of their generating function differ. For Model 22 we can solve for the generating function for all boundary conditions.
\begin{thm}\label{thm 22}
 For Model 22 the generating function $G(t;a,b)$ is a ratio of D-finite functions which simplifies to a D-finite function if  $a=1$ or $b=1$.
\end{thm}
\begin{proof}
Model 22 is isomorphic to the problem studied in \cite{tabbara2013}.
\end{proof}
\begin{table}[h!]
\begin{center}
\begin{tabular}{|c|m{11mm}|c|c|c|c|c|c|c|c|}
  \hline
  & \multicolumn{1}{c|}{Model} & (1,1) & (a,a) & (a,1) & (1,b) & (a,b) & Group\\
  \hline
  22 & \vspace{4pt}\includegraphics[height=1cm]{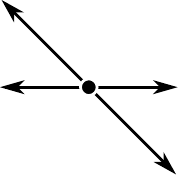} &
  DF  \cmark& DAlg \cmark & DF \cmark & DF \cmark & DAlg \cmark & 8: $(\olx y,y), (x,x^2\oly)$\\
  \hline
  23 & \vspace{4pt}\includegraphics[height=1cm]{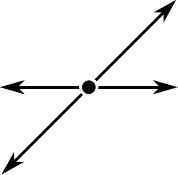} &
  Alg  \cmark& ? & Alg \xmark & DF \xmark & ? & 8: $(\olx \oly,y), (x, \olx^2\oly)$\\
  \hline
\end{tabular}
\end{center}
\caption{Step sets associated with the group generated by the involutions given in the final column that generate a group isomorphic to $D_4$.}
\label{tab 2223}
\end{table}
Unfortunately for Model 23 we could not derive a solution for any boundary values other than the original case of $a=b=1$.
However, our series analysis leads us to suggest that the solution stays algebraic if $b=1$ but may be D-finite when $a=1$.

The organisation of the remainder of this paper is as follows. In Section~\ref{sec:general_fe} we lay out the general functional equation which will be the starting point for all 23 models. In Sections~\ref{sec soln 134} and~\ref{sec soln 2} we solve Models 1--4. All horizontally symmetric models (5--16) are then considered together in Section~\ref{sec soln 5to16}. Models 17--19 are the focus of Sections~\ref{sec soln 17}--\ref{sec soln 19}, respectively. A final discussion and some open questions are presented in Section~\ref{sec:discussion}.

\section{The general functional equation}\label{sec:general_fe}
For a given model, let $\mathcal S \subseteq \{-1,0,1\}^2\setminus\{(0,0)\}$ be the set of allowable steps. The \emph{step set generator} is
\begin{equation}
S(x,y) = \sum_{(i,j)\in\mathcal S} x^i y^j.
\end{equation}
We write, following the notation from \cite{bousquet-melou_walks_2010}, that
\begin{equation}
S(x,y) = A_{-1}(x)\oly + A_0(x) + A_1(x)y = B_{-1}(y)\olx + B_0(y) + B_1(y)x.
\end{equation}

In the absence of any boundary weights, a functional equation for $Q(t;x,y;a,b) \equiv Q(x,y)$ can be obtained by considering how one obtains all walks by the addition of a single step onto a walk of any length.
Away from the boundaries adding a step gives  $tS(x,y)Q(x,y)$ and the zero length walk is counted by a $1$.
One cannot step south across the horizontal boundary so one needs to remove $t\oly A_{-1}(x) Q(x,0)$.
Similarly one cannot step west across the vertical boundary so one needs to subtract $t\olx B_{-1}(y)Q(0,y)$.
Finally note that one has subtracted twice the contribution of walks stepping south-east from the origin and so we must add them back once as $t\olx\oly\epsilon Q(0,0)$, where $\epsilon = [x^{-1}y^{-1}]\{S(x,y)\}$ and so is 1 if $(-1,-1)\in\mathcal S$ and 0 otherwise.
So we obtain
\begin{equation}\label{eqn:general_fe}
Q(x,y) = 1 + tS(x,y)Q(x,y) - t\oly A_{-1}(x) Q(x,0) - t\olx B_{-1}(y)Q(0,y) + t\olx\oly\epsilon Q(0,0).
\end{equation}
This is often written as
\begin{equation}
xyK(x,y)Q(x,y) = xy - txA_{-1}(x)Q(x,0) - tyB_{-1}(y)Q(0,y) + t\epsilon Q(0,0),
\end{equation}
where $K(x,y)$ is the \emph{kernel} and is given by
\begin{align} \label{eqn kern defn}
  K(x,y) &= 1-tS(x,y).
\end{align}

With boundary weights, things become more complicated. We must account for the following:
\begin{enumerate}[(1)]
\item when a walk steps along the $x$-axis, but not to $(0,0)$, it accrues weight $a$;
\item when a walk steps south onto the $x$-axis, but not to $(0,0)$, it accrues weight $a$;
\item when a walk steps along the $y$-axis, but not to $(0,0)$, it accrues weight $b$;
\item when a walk steps west onto the $y$-axis, but not to $(0,0)$, it accrues weight $b$; and
\item when a walk steps onto the vertex at $(0,0)$, it accrues weight $ab$.
\end{enumerate}

We define a more general indicator function: $\epsilon_{ij}$ is 1 if $(i,j) \in \mathcal S$, and is 0 otherwise, where $+$ stands for $1$ and $-$ stands for $-1$.
We address the five points above in turn.

\begin{enumerate}[(1)]
\item To~\eqref{eqn:general_fe}, we add
\begin{equation}
t(a-1)\epsilon_{+0}xQ(x,0) + t(a-1)\epsilon_{-0}\olx\left(Q(x,0) - Q(0,0)\right) - t(a-1)\epsilon_{-0}[x^1]\left\{Q(x,0)\right\}.
\end{equation}
We note that there is a factor of $(a-1)$ rather than just $a$ since all walks have already been enumerated but without the new weight.
Hence we subtract off the incorrectly weighted walks and then add them back with the correct new weight.
The first term accounts for walks moving east along the horizontal boundary, while the second term accounts for walks moving west along the horizontal boundary.
It is more complicated because of the conditions around the origin. The next four points follow by similar arguments.

\item Then add
\begin{multline}
t(a-1)\epsilon_{+-}x[y^1]\left\{Q(x,y)\right\} + t(a-1)\epsilon_{0-}[y^1]\left\{Q(x,y)-Q(0,y)\right\} \\ + t(a-1)\epsilon_{--}\olx[y^1]\left\{Q(x,y) - Q(0,y)\right\} -t(a-1)\epsilon_{--}[x^1y^1]\left\{Q(x,y)\right\}.
\end{multline}

\item Next add
\begin{equation}
t(b-1)\epsilon_{0+}yQ(0,y) + t(b-1)\epsilon_{0-}\oly\left(Q(0,y) - Q(0,0)\right) - t(b-1)\epsilon_{0-}[y^1]\left\{Q(0,y)\right\}.
\end{equation}

\item Then
\begin{multline}
t(b-1)\epsilon_{-+}y[x^1]Q(x,y) + t(b-1)\epsilon_{-0}[x^1]\left\{Q(x,y)-Q(x,0)\right\} \\ + t(b-1)\epsilon_{--}\oly[x^1]\left\{Q(x,y) - Q(x,0)\right\} - t(b-1)\epsilon_{--}[x^1y^1]\left\{Q(x,y)\right\}.
\end{multline}

\item Finally, add
\begin{multline}
t(ab-1)\epsilon_{-0}[x^1]\left\{Q(x,0)\right\} + t(ab-1)\epsilon_{0-}[y^1]\left\{Q(0,y)\right\} \\ + t(ab-1)\epsilon_{--}[x^1y^1]\left\{Q(x,y)\right\}.
\end{multline}
\end{enumerate}

Adding all these contributions gives an equation with 9 unknown terms.
While this is very unwieldy to write down (so we won't), it is relatively simple to manipulate using computer algebra.
Fortunately, 5 of these unknowns can be eliminated and the equation simplifies considerably, by using some additional relations.
These relations are effectively obtained by extracting the coefficient of $x^0$, $y^0$ or $x^0y^0$, or equivalently by considering those walks that either end on a boundary or at the origin.

Firstly, a walk ending on the $x$-axis was previously either on the $x$-axis or immediately above it (or is the empty walk), so
\begin{multline}
Q(x,0) = 1 + ta\epsilon_{+0}xQ(x,0) + ta\epsilon_{-0}\olx\left(Q(x,0)-Q(0,0)\right) + ta(b-1)\epsilon_{-0}[x^1]\left\{Q(x,0)\right\}
\\ + ta\epsilon_{+-}x[y^1]\left\{Q(x,y)\right\} + ta\epsilon_{0-}[y^1]\left\{Q(x,y)+(b-1)Q(0,y)\right\} \\ + ta\epsilon_{--}\olx[y^1]\left\{Q(x,y) - Q(0,y)\right\} + ta(b-1)\epsilon_{--}[x^1y^1]\left\{Q(x,y)\right\}.
\end{multline}

Similarly for walks ending on the $y$-axis:
\begin{multline}
Q(0,y) = 1 + tb\epsilon_{0+}yQ(0,y) + tb\epsilon_{0-}\oly\left(Q(0,y)-Q(0,0)\right) + t(a-1)b\epsilon_{0-}[y^1]\left\{Q(0,y)\right\}
\\ + tb\epsilon_{-+}y[x^1]\left\{Q(x,y)\right\} + tb\epsilon_{-0}[x^1]\left\{Q(x,y)+(a-1)Q(x,0)\right\} \\ + tb\epsilon_{--}\oly[x^1]\left\{Q(x,y) - Q(x,0)\right\} + t(a-1)b\epsilon_{--}[x^1y^1]\left\{Q(x,y)\right\}.
\end{multline}

Finally, for walks ending at the corner:
\begin{equation}\label{eqn:general_Q00_fe}
Q(0,0) = 1 + tab\epsilon_{-0}[x^1]\left\{Q(x,0)\right\} + tab\epsilon_{0-}[y^1]\left\{Q(0,y)\right\} + tab\epsilon_{--}[x^1y^1]\left\{Q(x,y)\right\}.
\end{equation}

Using all of these equations and eliminating all the unknowns except $Q(x,y)$, $Q(x,0)$, $Q(0,y)$ and $Q(0,0)$ we finally arrive at a functional equation satified by $Q(x,y)$.
\begin{thm}
The generating function $Q(x,y)$, with weight $a$ associated with vertices on the $x$-axis and weight $b$ associated with vertices on the $y$-axis, satisfies the functional equation
\begin{multline}\label{eqn:simplified_general_fe}
xyK(x,y)Q(x,y) = \frac{xy}{ab} + x\left(y - \frac{y}{a} - t A_{-1}(x)\right)Q(x,0) + y\left(x - \frac{x}{b} - t B_{-1}(y)\right)Q(0,y) \\ - \left(\frac{xy}{ab}\left(1-a\right)\left(1-b\right) - t\epsilon\right)Q(0,0).
\end{multline}
\end{thm}

\begin{proof}
We can proceed as described above. Alternatively, we can derive the equation via a double-counting argument, which we give here.

Let
\begin{equation}
Q^\dagger(x,y) = Q(x,y) + \frac1a Q(x,0) + \frac1b Q(0,y) + \left(\frac{1}{ab}-\frac1a-\frac1b\right)Q(0,0).
\end{equation}
Observe that $Q^\dagger(x,y)$ counts every walk ending on a boundary twice: once in $Q(x,y)$ (with the correct weight), and once in the other terms (underweighted by exactly one factor of $a$, $b$ or $ab$, depending on whether the walk ends on the $x$-axis, the $y$-axis, or in the corner). Every other walk is counted once with the correct weight (in $Q(x,y)$).

Let us now count this same set in another way.
Firstly, walks ending on the boundaries (with the correct weights) have generating function
\begin{equation}\label{eqn:weirdsum_part1}
Q(x,0) + Q(y,0) - Q(0,0).
\end{equation}
Next, a non-empty walk can be constructed by appending a step to an existing walk, while being careful to subtract off walks that overstep the boundaries:
\begin{equation}\label{eqn:weirdsum_part2}
t\left(S(x,y)Q(x,y) - \oly A_{-1}(x)Q(x,0) - \olx B_{-1}(y)Q(0,y) + \olx\oly\epsilon Q(0,0)\right).
\end{equation}
Note that a walk counted by~\eqref{eqn:weirdsum_part2} which ends on a boundary will be missing exactly one factor of $a$, $b$ or $ab$, depending on whether it finishes on the horizontal boundary, vertical boundary or at the corner (respectively). Walks which do not end on a boundary are correctly weighted.

Finally, the empty walk has not yet been counted with the incorrect weight -- this is simply $\frac{1}{ab}$. Adding this term to~\eqref{eqn:weirdsum_part1} and~\eqref{eqn:weirdsum_part2}, and equating with $Q^\dagger(x,y)$, gives
\begin{align}
Q^\dagger(x,y) &= Q(x,y) + \frac1a Q(x,0) + \frac1b Q(0,y) + \left(\frac{1}{ab}-\frac1a-\frac1b\right)Q(0,0) \notag \\
&=
\begin{multlined}[t]
  \frac{1}{ab} + Q(x,0) + Q(y,0) - Q(0,0) \\ + t\left(S(x,y)Q(x,y) - \oly A_{-1}(x)Q(x,0) - \olx B_{-1}(y)Q(0,y) + \olx\oly\epsilon Q(0,0)\right).
\end{multlined}
\end{align}
Multiplying both sides by $xy$ and rearranging gives~\eqref{eqn:simplified_general_fe}.
\end{proof}

Notice that equation~\eqref{eqn:simplified_general_fe} neatly separates into a bulk term, involving $Q(x,y)$, on the left-hand side, and initial and boundary terms, involving $Q(x,0), Q(0,y)$ and $Q(0,0)$, on the right-hand side.
To solve these equations for different step sets we use the obstinate kernel method \cite{bousquet2002counting}.
For almost all models we use a half-orbit sum over selected symmetries of the kernel, which  allows us to remove selected boundary terms.
Then setting the kernel to zero by choosing an appropriate root then eliminates the bulk terms.
We then extract the power series $Q(0,0)$ by calculating a constant term.
Model 19, Kreweras walks, requires a modified approach --- see \cite{bousquet2005} and Section~\ref{sec soln 19}.

\section{Models 1, 3 \& 4}
\label{sec soln 134}

Models 1--4 all have vertical and horizontal symmetry.
It turns out that Model 2 has an additional property which simplifies its solution, so we will address that case separately.
For brevity we will focus here on Model 1, as 3 and 4 can be solved using very similar arguments.

\subsection{Model 1}
Here we have $S(x,y) = x + y + \olx + \oly$, so that $A_{-1}(x) = B_{-1}(y) = 1$ and $\epsilon = 0$. Equation~\eqref{eqn:simplified_general_fe} becomes
\begin{multline}\label{eqn:nsew_main_fe}
xyK(x,y)Q(x,y) = \frac{xy}{ab} + x\left(y - \frac{y}{a}-t\right)Q(x,0) \\
  + y\left(x-\frac{x}{b}-t\right)Q(0,y) - \frac{xy}{ab}(1-a)(1-b)Q(0,0),
\end{multline}
where the kernel is
\begin{align}
K(x,y) &= 1-t(x+y+\olx+\oly).
\end{align}
As noted in Table~\ref{tab 1to4}, the kernel is symmetric under the involutions $(x,y)\mapsto (\olx,y)$ and $(x,y)\mapsto (x,\oly)$.
These generate a group of 4 symmetries $\{(x,y), (\olx,y), (x,\oly), (\olx,\oly) \}$.
The latter three symmetries can be used to construct new equations.
We will, shortly, set the kernel to zero by substituting a power series for $y$.
Because of this, we will only make use of $\{(x,y), (\olx,y) \}$, since the power series that then result from substituting a power series for $y$ are convergent. If we use $\{(x,\oly), (\olx,\oly) \}$, then we cannot make the same substitution for $y$ since the results do not converge inside the ring of formal power series.

Substituting $(x,y)\mapsto (\olx,y)$ into our functional equation gives
\begin{multline}\label{eqn:nsew_reflected_fe}
\olx yK(x,y)Q(\olx,y) = \frac{\olx y}{ab} + \olx\left(y - \frac{y}{a}-t\right)Q(\olx,0) \\
+ y\left(\olx-\frac{\olx}{b}-t\right)Q(0,y) - \frac{\olx y}{ab}(1-a)(1-b)Q(0,0).
\end{multline}
We then eliminate $Q(0,y)$ by taking an appropriate linear combination of equations~\eqref{eqn:nsew_main_fe} and~\eqref{eqn:nsew_reflected_fe},
\begin{multline}\label{eqn:nsew_fe_after_halforbit}
\frac{ayK(x,y)}{ay-ta-y}\left(Q(x,y)+\frac{\olx(bx-tb-x)}{tbx-b+1}Q(\olx,y)\right)  = \frac{t\olx y(x^2-1)}{(ay-ta-y)(tbx-b+1)} + Q(x,0) \\ + \frac{\olx(bx-tb-x)}{tbx-b+1}Q(\olx,0) -\frac{t\olx y(a-1)(b-1)(x^2-1)}{(ay-ta-y)(tbx-b+1)}Q(0,0).
\end{multline}

Let $Y \equiv Y(t;x)$ be the root of $K(x,y)$ in the variable $y$ which is a power series in $t$.
\begin{align}
  Y(t;x) &= \frac{x-t-tx^2 - \sqrt{(t-x-2tx+tx^2)(t-x+2tx+tx^2)}}{2tx} = t+ O(t^2).
\end{align}
Since all unknowns in~\eqref{eqn:nsew_fe_after_halforbit} are power series in $t$ with coefficients that are polynomials in $y$, the substitution $y\mapsto Y$ is valid (ie.~all functions converge as power series), and cancels the LHS of~\eqref{eqn:nsew_fe_after_halforbit}.

Thus
\begin{multline}\label{eqn:nsew_fe_after_halforbit_cancelkernel}
0 = \frac{t\olx Y(x^2-1)}{(aY-ta-Y)(tbx-b+1)} + Q(x,0) \\ + \frac{\olx(bx-tb-x)}{tbx-b+1}Q(\olx,0) -\frac{t\olx Y(a-1)(b-1)(x^2-1)}{(aY-ta-Y)(tbx-b+1)}Q(0,0).
\end{multline}
Notice that the coefficient of $Q(\olx,0)$ is a rational function of $b,x$ and $t$.
This allows us to compute a required constant term in closed form as we now demonstrate.

\subsubsection{General $a,b$ directly}\label{sssec:nsew_general_ab_direct}
Notice that the first and last terms in equation~\eqref{eqn:nsew_fe_after_halforbit_cancelkernel} are nearly identical. Grouping them together gives
\begin{multline}\label{eqn:nsew_rearranged_slightly}
0 = \frac{t\olx Y(x^2-1)}{(aY-ta-Y)(tbx-b+1)} \Big(1 - (a-1)(b-1)Q(0,0) \Big) \\
+ Q(x,0) + \frac{\olx(bx-tb-x)}{tbx-b+1}Q(\olx,0).
\end{multline}
Write
\begin{align}
  C(t) &= [x^0] \left\{ \frac{t\olx Y(x^2-1)}{(aY-ta-Y)(tbx-b+1)} \right\},
\end{align}
then the coefficient of $x^0$ of the first term is
\begin{equation}\label{eqn:nsew_ct_simple_bits}
C(t) \left(1 - (a-1)(b-1)Q(0,0)\right).
\end{equation}
The constant term of the second term is just $Q(0,0)$.

As remarked above, the rationality of the coefficient of $Q(\olx,0)$ allows us to compute the required constant term.
The second term can be viewed as the product of two power series in $t$:
\begin{align}
\frac{bx-bt-x}{(bxt - b + 1)x}
&= -\frac{1}{x^2} + \frac{1-x^2}{x^2} \sum_{n\geq 0} \left(\frac{bx}{b-1}\right)^n t^n
\nonumber \\
&= -\frac{1}{x^2} + \sum_{n \geq 0} \left(\frac{tb}{b-1} \right)^n (x^n - x^{n-2}) \\
\intertext{and}
Q(\olx,0) &= \sum_{m,k\geq 0} q_{m,k,0} t^m x^{-k}.
\end{align}
When we take the product and extract the constant term, the only terms that are left are
\begin{multline}
- \sum_{m,k\geq0} q_{m,k,0} t^n \left(\frac{tb}{b-1}\right)^k
+ \sum_{m,k\geq0} q_{m,k,0} t^n \left(\frac{tb}{b-1}\right)^{k+2} \\
= -\left(1 - \left(\frac{bt}{b-1}\right)^2 \right) Q\left(\frac{tb}{b-1},0 \right).
\end{multline}
Reassembling these components we have
\begin{align}\label{eqn after ct}
  0 &= C(t) \left(1 - (a-1)(b-1)Q(0,0)\right) + Q(0,0) - \left(1 - \left(\frac{bt}{b-1}\right)^2 \right) Q\left(\frac{tb}{b-1},0 \right).
\end{align}

At this point, it appears that we have just replaced $Q(\olx,0)$ with a new unknown $Q\left(\frac{tb}{b-1},0\right)$.
In fact, this new unknown can be expressed in terms of $Q(0,0)$ by careful manipulation of the original functional equation.
Substituting $x\mapsto \frac{tb}{b-1}$ in~\eqref{eqn:nsew_main_fe} we get
\begin{multline}
\frac{tby}{b-1}K\left(\frac{tb}{b-1},y\right)Q\left(\frac{tb}{b-1},y\right) \\
= \frac{ty}{a(b-1)} + \frac{tb}{b-1}\left(y-\frac{y}{a}-t\right)Q\left(\frac{tb}{b-1},0\right) + \frac{ty}{a}(1-a)Q(0,0).
\end{multline}
Note, in particular, that the $Q(0,y)$ term has been cancelled.
The new kernel $K\left(\frac{tb}{b-1},y\right)$ has a power series root in $y$:
\begin{align}
  Y'(t;b) &= Y\left(t; \frac{tb}{b-1}\right).
\end{align}
Substituting $y \mapsto Y'$, we have
\begin{equation}\label{eqn:nsew_cancel_second_kernel}
0 = \frac{Y'}{a(b-1)} + \frac{b}{b-1}\left(Y'-\frac{Y'}{a}-t\right)Q\left(\frac{tb}{b-1},0\right) + \frac{Y'}{a}(1-a)Q(0,0).
\end{equation}
Now isolate $Q\left(\frac{tb}{b-1},0\right)$:
\begin{align}
  Q\left(\frac{tb}{b-1},0\right) &= -\frac{Y'}{b(aY'-Y'-ta)} \left( 1- (a-1)(b-1) Q(0,0) \right) \\
  & \equiv -P \cdot \left( 1- (a-1)(b-1) Q(0,0) \right).
\end{align}
Substituting this into~\eqref{eqn after ct} and solving for $Q(0,0)$ gives
\begin{align}
G(t;a,b) &= Q(0,0) \\ &=
\begin{multlined}[t]
\frac{1}{(a-1)(b-1)} \\
+ \frac{1}{(a-1)(b-1) + (a-1)^2(b-1)^2 C - (a-1)^2(tb-b+1)(tb+b-1) P}.
\end{multlined}
\end{align}
The first few terms are
\begin{multline}
Q(0,0) = 1 + (a^2 b + a b^2) t^2 \\ + (a^3 b + a^4 b + 2 a^2 b^2 + a^4 b^2 +
    a b^3 + 2 a^3 b^3 + a b^4 + a^2 b^4) t^4 + O(t^6).
\end{multline}
Since $P$ is algebraic and $C$ is D-finite, $Q$ is at worst D-algebraic.

\subsubsection{General $a,b$ via $a=1$ and $b=1$}\label{sssec:nsew_general_ab_via_a1b1}
We now present a similar solution to the model which uncovers more structure of the problem, at the expense of being less direct.
In particular we find an expression for $G(t;a,b)$ in terms of $G(t;a,1)$ and $G(t;1,b)$.
These two subproblems are significantly easier to solve. Let $P(x,y) = Q(x,y)|_{b=1}$ and $R(x,y) = Q(x,y)|_{a=1}$.
We first solve for $P$ and then $R$ follows immediately by symmetry.

When $b=1$ the functional equation and its solution simplify dramatically. Equation~\eqref{eqn:nsew_fe_after_halforbit_cancelkernel} becomes
\begin{equation}\label{eqn nick ref this}
0 = \frac{\olx^2 Y(x^2-1)}{aY-ta-Y} + P(x,0) \\ - \olx^2P(\olx,0)
\end{equation}
and hence
\begin{equation}\label{eqn:nsew_b=1}
P(0,0) = -[x^0]\left\{\frac{\olx^2 Y(x^2-1)}{aY-ta-Y}\right\}.
\end{equation}
This is a D-finite function. By symmetry we then have the solution for $R(0,0)$.

We now show that the solution for general $a,b$ can be written in terms of $P$ and $R$. Rearrange~\eqref{eqn:nsew_fe_after_halforbit_cancelkernel} once more:
\begin{multline}
0 = \frac{t\olx^2 Y(x^2-1)}{(aY-ta-Y)} + \olx(tbx-b+1)Q(x,0) \\ + \olx^2(bx-tb-x)Q(\olx,0) -\frac{t\olx^2 Y(a-1)(b-1)(x^2-1)}{aY-ta-Y}Q(0,0).
\end{multline}
Notice that the $Q$-independent term is precisely $t$ times the $P$-independent term in~\eqref{eqn nick ref this}.
Hence its constant term is exactly $t P(0,0)$. Consequently the constant term of the whole equation is
\begin{equation}\label{eqn 1P}
0 = -tP(0,0) + tbQ(0,0) - (b-1)[x^1]\left\{Q(x,0)\right\} + t(a-1)(b-1)P(0,0)Q(0,0).
\end{equation}
By symmetry,
\begin{equation}\label{eqn 1R}
0 = -tR(0,0) + taQ(0,0) - (a-1)[y^1]\left\{Q(0,y)\right\} + t(a-1)(b-1)R(0,0)Q(0,0).
\end{equation}
At this point we have introduced two new unknowns $[x^1]\left\{Q(x,0)\right\}$ and $[y^1]\left\{Q(0,y)\right\}$.
These can be eliminated using the following relation
\begin{equation}
Q(0,0) = 1 + tab[x^1]\left\{Q(x,0)\right\} + tab[y^1]\left\{Q(0,y)\right\},
\end{equation}
which is equation~\eqref{eqn:general_Q00_fe} for Model 1.
This equation simply says that any walk ending in the corner either has length 0 or arrives by taking a step south or west.
This system of three equations can be solved to find $Q(0,0)$ in terms of $P\equiv P(0,0)$ and $R\equiv R(0,0)$:
\begin{multline}
  Q(0,0) = \\
\frac{(1 - a)(1-b) - t^2ab\left((a-1)P + (b-1)R\right)}{(1 - a)(1-b) - t^2ab(2ab-a-b) - t^2ab(a-1)(b-1)\left((a-1)P + (b-1)R\right)}.
\end{multline}
When $a=1$ or $b=1$, this simply reduces to $R(0,0)$ and $P(0,0)$ respectively.
Also the above expression can be written in partial fraction form so that the numerator is free of $P$ and $R$.
The solution of Model~22 \cite{tabbara2013} has a very similar form.

\subsection{Models 3 \& 4}

The methodology from Section~\ref{sssec:nsew_general_ab_direct} above can be used to solve Models 3 and 4.
Things are complicated somewhat by the fact that the denominator of the coefficient of $Q(\olx, 0)$ in the equivalent of~\eqref{eqn:nsew_fe_after_halforbit_cancelkernel} is quadratic in $x$.
The method still enables us to find $Q(0,0)$ as the ratio of two D-finite functions, and so at worst the solution is D-algebraic.
Moreover, at $a=1$ or $b=1$, the solution is D-finite, for the same reason as for Model 1. We do not know how to apply the methodology of Section~\ref{sssec:nsew_general_ab_via_a1b1} to Models 3 and 4, again due to the presence of that quadratic term. This term leads to extra unknowns in the equivalents of~\eqref{eqn 1P} and~\eqref{eqn 1R}.


\section{Model 2}
\label{sec soln 2}

Model 2 has horizontal and vertical symmetry but is different from Models 1, 3 and 4 in that it has a D-finite generating function for all $a,b$. In contrast, the solution method, via functional equations is actually more complicated.

There is, however, a very elegant solution that comes from the observation that $G(t;1,1)$ can be written as the Hadamard product of the generating function of Dyck paths with itself. To see this, note that any walk starting and ending at the origin can be factored into a pair of Dyck paths. More precisely, the $x$-coordinate (and $y$-coordinate) for any such walk of length $2n$ is a sequence of non-negative integers starting and ending at $0$, with steps $\pm 1$.
The number of such sequences is given by the $n^\mathrm{th}$ Catalan number $c_n = \frac{1}{n+1}\binom{2n}{n}$.
Since these two sequences are completely independent we have that the number of such paths (ignoring, for the moment visits to either wall) is just $c_n^2$.

This reasoning continues to hold when visits to either wall are counted, and one obtains
\begin{align}
  [t^n]\left\{Q(0,0)\right\} = G(t;a,b) &= Z_{2n}(a) Z_{2n}(b),
\end{align}
where $Z_{2n}(a)$ is the coefficient of $t^n$ in
\begin{align}
  f(t;a) &= \frac{2}{2-a-a\sqrt{1-4t^2}} = \sum_{n\geq 0} t^{2n} \sum_{k=1}^n \frac{k}{n} \binom{2n-k-1}{n-k}a^k,
\end{align}
being the generating function of Dyck paths counted by their length and number of visits to the axis.
This generating function is classical and can be quickly established via a factorisation argument which gives the functional equation
\begin{align}
  f(t;a) &= 1 + t^2 f(t;1) f(t;a).
\end{align}
Consequently the generating function $Q(0,0)$ can be written as the Hadamard product
\begin{align}
  Q(0,0) = G(t;a,b) &= f(t;a) \odot_t f(t;b) = \sum_{n\geq0} t^n \left([t^n]\{f(t,a)\} \right)\left([t^n]\{f(t,b)\} \right).
\end{align}
Since D-finite functions are closed under Hadamard products \cite{lipshitz1988}, the above is D-finite.
Note that it cannot be algebraic, because when $a,b=1$, the coefficients of $Q(0,0)$ grow as $16^n/n^3$, a form that is incompatible with algebraic generating functions --- see, for example, Section VII.7 of \cite{flajolet2009}.

We can also show that the solution is D-finite via the functional equation using the kernel method.
The functional equation is
\begin{multline}
xyK(x,y)Q(x,y) = \frac{xy}{ab} + x\left(y - \frac{y}{a} - t(x+\olx)\right)Q(x,0) \\
+ y\left(x - \frac{x}{b} - t(y+\oly)\right)Q(0,y)  - \left(\frac{xy}{ab}\left(1-a\right)\left(1-b\right) - t\right)Q(0,0),
\end{multline}
where
\begin{equation}
K(x,y) = 1-t(xy + x\oly + \olx y+\olx\oly).
\end{equation}
As noted in Table~\ref{tab 1to4}, the kernel is invariant under the involutions $(x,y)\mapsto(\olx,y), (x,\oly)$.
Proceeding as for Model 1, we map $(x,y)\mapsto (\olx,y)$,
\begin{multline}
\olx yK(x,y)Q(\olx,y) = \frac{\olx y}{ab} + \olx\left(y - \frac{y}{a} - t(x+\olx)\right)Q(\olx,0)\\
 + y\left(\olx - \frac{\olx}{b} - t(y+\oly)\right)Q(0,y) - \left(\frac{\olx y}{ab}\left(1-a\right)\left(1-b\right) - t\right)Q(0,0).
\end{multline}
Eliminating $Q(0,y)$ between the two functional equations we have,
\begin{multline}
\frac{a x yK(x,y)}{ta + ta x^2+x y-a x y}\left(-Q(x,y)
+\frac{\olx(tb+x y-b x y+tb y^2)}{tb x+y-b y+tb x y^2}Q(\olx,y)\right)
\\ = -\frac{ty(x^2-1)(1+y^2)}{(ta+ta x^2+x y-a x y) (tb x+y-b y+tb x y^2)} + Q(x,0) \\ - \frac{\olx(tb+x y-b x y+tb y^2)}{tb x+y-b y+tb x y^2}Q(\olx,0) \\
+ \frac{ty(b-1) (x^2-1)(ay^2-y^2-1)}{(ta+ta x^2+x y-a x y) (tb x+y-b y+tb x y^2)}Q(0,0).
\end{multline}

As for Model 1, we set the kernel to zero by substituting $y=Y(t;x)$ being a power series solution to $K(x,y)=0$.
After simplifying the coefficients we arrive at
\begin{multline}\label{eqn:diag_after_orbit_sum}
0 = -\frac{xY(x^2-1)}{(1+x^2-b) (ta+ta x^2+x Y-a xY)}\\
 + Q(x,0) -\frac{1+x^2-b x^2}{1+x^2-b}Q(\olx,0)  -\frac{(b-1) (x^2-1)}{1+x^2-b}Q(0,0).
\end{multline}

The first term in~\eqref{eqn:diag_after_orbit_sum} can be viewed as a power series in $t$ with coefficients that are Laurent polynomials in $x$; it thus has a well-defined constant term with respect to $x$ (see~\eqref{eqn:diag_soln_from_CT_xb} below).

Now the coefficient of $Q(\olx,0)$ does not depend on $t$, so expanding everything as a power series in $t$ and taking the constant term with respect to $x$ is not an option here. We are left to consider things as in terms of $x$ or, alternatively, $\olx$.

Expanding in $\olx$, things are straightforward. Since
\begin{align}
-\frac{1+x^2-b x^2}{1+x^2-b} &= (b-1) + b(b-2)\olx^2 + O(\olx^4), \text{ and} \\
-\frac{(b-1) (x^2-1)}{1+x^2-b} &= -(b-1) - (b-1)(b-2)\olx^2 + O(\olx^4),
\end{align}
the third and fourth terms cancel, leaving
\begin{equation}\label{eqn:diag_soln_from_CT_xb}
Q(0,0) = [x^0]\left\{\frac{xY(x^2-1)}{(1+x^2-b) (ta+ta x^2+x Y-a xY)}\right\}.
\end{equation}
This shows that $Q(0,0)$ is D-finite.
We note that if one expands in $x$ rather than $\olx$ then one can arrive, with more work, at a similar though less appealing D-finite expression.

\section{D-finite solution to horizontally symmetric models}\label{sec soln 5to16}
Models 1 through 16 all possess symmetry across a vertical line, and we have stated that all of them are D-finite when $b=1$.
We prove that this is the case using a variation of Bousquet-M\'elou's argument from \cite{bousquet2002counting}.
Let us start by giving the functional equation satisfied by these models when $b=1$:
\begin{align}
xyK(x,y)Q(x,y) = \frac{xy}{a} + x\left(y - \frac{y}{a} - t A_{-1}(x)\right)Q(x,0) - ty B_{-1}(y)Q(0,y) + t\epsilon Q(0,0).
\end{align}

Since the step set is symmetric across a vertical line, the kernel is symmetric under the involution $(x,y) \mapsto (\olx,y)$.
Using that substitution we obtain
\begin{align}
  \olx y K(x,y)Q(\olx,y) = \frac{y}{ax} + \olx\left(y - \frac{y}{a} - t A_{-1}(x)\right)Q(\olx,0) - ty B_{-1}(y)Q(0,y) + t\epsilon Q(0,0),
\end{align}
where we have used the fact that $A_{-1}(\olx)=A_{-1}(x)$ (thanks to the symmetry of the step set).
Subtracting one equation from the other eliminates the $Q(0,y)$ term:
\begin{align}
    K(x,y) \left(xy Q(x,y) - \olx y Q(\olx,y) \right)
    &= \frac{y(x-\olx)}{a} + \left(y - \frac{y}{a} - t A_{-1}(x) \right)\left(xQ(x,0) - \olx Q(\olx,0)\right).
\end{align}
We can now set the kernel to zero by setting $y = Y(t;x) \equiv Y(x)$, being the power series solution of $K(x,y)=0$.
After a little rearrangement we have
\begin{align}
  Q(x,0) - \olx^2 Q(\olx,0)
  &= \frac{Y(x) (1-\olx^2)}{Y(x)(1-1/a) - t A_{-1}(x)}.
\end{align}
Now we are in familiar territory and we proceed by extracting constant term with respect to $x$ from both sides:
\begin{align}
    Q(0,0)=G(t;a,1) &= [x^{0}]\left\{ \frac{Y(x) (1-\olx^2)}{at A_{-1}(x) + (1-a)Y(x)} \right\}.
\end{align}
This shows that $Q(0,0)$ is D-finite.
We note that when $a=b=1$ none of Models 1 through 16 have an algebraic generating function, and so none of these models have an algebraic generating function when $b=1$.

\section{Model 17}\label{sec soln 17}
The functional equation and kernel of Model 17 are:
\begin{align}
  \label{eqn 17original}
  abxy K(x,y) Q(x,y)  &=
\begin{multlined}[t]
  xy
  + bx\left(ay - y - axt \right)Q(x,0)  \\
  + ay\left(bx - x - bt \right)Q(0,y)
  - (a-1)(b-1)xy Q(0,0),
\end{multlined}\\
\text{where }
K(x,y) &= 1-t(y + \olx  + x\oly ).
\end{align}

As noted in Table~\ref{tab 1718}, the kernel is symmetric under involutions $(x,y) \mapsto (\olx y,y)$ and $(x,y) \mapsto (x,x\oly)$.
This gives a group of 6 symmetries, which generate 5 new equations. Let us start by using those obtained by setting $(x,y) = (\olx y,y), (\olx y,\olx )$:
\begin{align}
  ab\olx y^2 K(\olx y,y) Q(x,y) &
\begin{multlined}[t]
 = \olx y^2
  + b\olx^2y^2\left(ax-x-at \right)Q(\olx y,0) \\
  + a\olx y\left(by-y-bxt \right)Q(0,y)
  - (a-1)(b-1) \olx y^2 Q(0,0)
\end{multlined}
\intertext{and}
  abxy K(x,y) Q(\olx y,\olx )&
  \begin{multlined}[t]
    = \olx ^2y
  + b\olx^2 y\left(a-1-ayt \right)Q(\olx y,0) \\
  + a\olx^2\left(by-y-bxt \right)Q(0,\olx )
  - (a-1)(b-1) \olx ^2y Q(0,0).
\end{multlined}
\end{align}
We can set the kernel to zero in all of these by choosing $Y = Y(x)$, being the power series solution to $K(x,y)=0$.
This leaves only the right-hand side of the 3 equations.
Then taking a linear combination of the equations we eliminate the $Q(0,y)$ and $Q(\olx y,0)$ terms, and after a little massaging of the result we have
\begin{multline}\label{eqn 17fq1}
Q(x,0) +
{\frac {a \left( tb-bx+x \right)  \left( at-ax+x \right) Y}{{x}^{2} \left( Yat-a+1 \right)  \left( -txa+Ya-Y
 \right) b} Q(0,\olx )} \\
= H(x,Y(x);a,b)\left( 1 - Q(0,0)(a-1)(b-1)\right)
\end{multline}
for a function $H$ which is rational in its arguments.

With a little algebra one can verify that the coefficient of $Q(0,\olx )$ is actually a rational function of $t,x,a,b$:
\begin{align}
  {\frac {a \left( tb-bx+x \right)  \left( at-ax+x \right) }{ \left( {x}^{2}{t}^{2}{a}^{2}+t{a}^{2}-at-ax+x \right) xb}}.
\end{align}

Before solving equation~\eqref{eqn 17fq1} in generality, let us first verify that when $a=1$ we obtain a D-finite solution.
\subsection{An aside to $a=1$}
Note that equation~\eqref{eqn 17fq1} simplifies considerably when $a=1$ to give
\begin{align}
  Q(x,0) + \frac{tbx-bx+x}{x^3}Q(0,\olx ) &=
  \frac{Y(Ybtx-btx^3-bt+bx-x)}{tbx^3(Yb-btx-Y)}.
\end{align}
Extracting the constant term with respect to $x$ of both sides then yields
\begin{align}
  \left.Q(0,0)\right|_{a=1} = G(t;1,b) &=
  [x^{0}]\left\{  \frac{Y(Ybtx-btx^3-bt+bx-x)}{tbx^3(Yb-btx-Y)}
  \right\}.
\end{align}
This demonstrates that when $a=1$ $Q(0,0)$ is D-finite.

\subsection{Back to general $(a,b)$}
We rewrite equation~\eqref{eqn 17fq1} with the coefficient of $Q(0,\olx )$ made explicitly rational:
\begin{multline}\label{eqn 17fq2}
Q(x,0) +
\frac {a \left( tb-bx+x \right)  \left( at-ax+x \right) }{ \left( {x}^{2}{t}^{2}{a}^{2}+t{a}^{2}-at-ax+x \right) xb}
Q(0,\olx ) \\
= H(x,Y(x);a,b)\left( 1 - Q(0,0)(a-1)(b-1)\right).
\end{multline}
We proceed as per the $a=1$ case, by taking the constant term of the above with respect to $x$. The main difficult now lies in computing
\begin{align}
[x^0] \left\{  \frac {a \left( tb-bx+x \right)  \left( at-ax+x \right) }{ \left( {x}^{2}{t}^{2}{a}^{2}+t{a}^{2}-at-ax+x \right) xb}
  Q(0,\olx ) \right\}.
\end{align}
Write the coefficient of $Q(0,\olx )$ as
\begin{align}\label{eqn 17 parfrac}
  C_0 + \frac{C_1}{x} + \frac{C_2}{1-x/\xi_+} + \frac{C_3}{1-\olx \cdot \xi_-},
\end{align}
where $\xi_\pm$ are solutions of the denominator factor ${x}^{2}{t}^{2}{a}^{2}+t{a}^{2}-at-ax+x$:
\begin{align}
  \xi_+ &= \frac{(a-1) + \sqrt{(1-a)(1-a+4a^3t^3)}}{2a^2t^2} = \frac{a-1}{a^2} t^{-2} + O(t); \\
  \xi_- &= \frac{(a-1) - \sqrt{(1-a)(1-a+4a^3t^3)}}{2a^2t^2} = at + O(t^4).
\end{align}
This particular combination of $x, \olx, \xi_\pm$ was chosen in~\eqref{eqn 17 parfrac} so that the expression is a power series in $t$.
The coefficient $C_1 = \frac{at}{a-1}$ while the others are somewhat messy algebraic functions of $a,b$ and $t$.
The desired constant term can now be written quite cleanly in terms of the $C_i$:
\begin{align}
[x^0] \left\{
\left( C_0 + \frac{C_1}{x} + \frac{C_2}{1-x/\xi_+} - \frac{C_0}{1-\olx \cdot \xi_-} \right)
Q(0,\olx ) \right\}
&= C_2 \cdot Q(0,1/\xi_+),
\end{align}
where we have used the fact that $C_0 = -C_3$.
This transforms equation~\eqref{eqn 17fq2} into
\begin{align}\label{eqn 17fq3}
  Q(0,0) +
  C_2 Q(0,1/\xi_+) &= \left( 1 - Q(0,0)(a-1)(b-1)\right) \cdot [x^0] \left\{ H(x,Y(x);a,b) \right\}.
\end{align}
Thankfully we can now compute $Q(0,1/\xi+)$ in terms of $Q(0,0)$ after a little work on the original functional equation~\eqref{eqn 17original}.

Notice that we can eliminate $Q(x,0)$ from equation~\eqref{eqn 17original} by setting $y = \frac{axt}{a-1}$. This yields
\begin{multline}
  abx K\left(x, \frac{axt}{a-1} \right) Q\left(x,\frac{axt}{a-1}\right)
  = \\
  x + a\left(bx - x - bt \right)Q\left(0,\frac{axt}{a-1}\right)
  - (a-1)(b-1)x Q(0,0).
\end{multline}
We now choose $x$ to set $K\left(x, \frac{axt}{a-1} \right)$ to zero. The solutions are precisely $x = \xi_\pm$.
Since $\xi_- = O(t)$ we set $x=\xi_-$ in \eqref{eqn 17fq3}. Notice that $\frac{a\xi_- t}{a-1} = 1/\xi_+$, and so we have (after a little manipulation)
\begin{align}
  Q(0,1/\xi_+) &= \frac{\xi_-}{a\left(b\xi_- - \xi_- - bt \right)} \left( (a-1)(b-1) Q(0,0) -1\right).
\end{align}
Substitute this back into equation~\eqref{eqn 17fq3}:
\begin{multline}
  Q(0,0) + \frac{C_2 \xi_-}{a\left(b\xi_- - \xi_- - bt \right)} \left( (a-1)(b-1) Q(0,0) -1\right)
  \\ = \left( 1 - Q(0,0)(a-1)(b-1)\right) \cdot [x^0] \left\{ H(x,Y(x);a,b) \right\}.
\end{multline}
Finally, we can isolate $Q(0,0)$:
\begin{align}
  Q(0,0) &= \frac{P(t;a,b)}{1 + (a-1)(b-1)P(t;a,b)},
\end{align}
where
\begin{align}
  P(t;a,b) &= \frac{C_2 \xi_-}{a\left((b-1)\xi_-  - bt \right)} - [x^0] \left\{ H(x,Y(x);a,b) \right\}.
\end{align}
Since $P(t;a,b)$ is D-finite, $Q(0,0)$ is, at worst, D-algebraic. Unfortunately this does not prove that $Q(0,0)$ is not D-finite.
Additionally, the series $P(t;a,b)$ is singular when $a=1$ or $b=1$.

\subsection{Aside to $b=1$}
While we cannot substitute $b=1$ directly into the above expression for $Q(0,0)$ we can recycle most of our workings.
Set $b=1$ into equation~\eqref{eqn 17fq2} and take the constant term:
\begin{align}
  Q(0,0) + [x^0]\left\{
  \frac {at \left( at-ax+x \right) }{ \left( {x}^{2}{t}^{2}{a}^{2}+t{a}^{2}-at-ax+x \right) x}
  Q(0,\olx ) \right\}
  & = [x^0] H(x,Y(x);a,1).
\end{align}
We can then compute the constant term on the left-hand side by very similar methods and demonstrate that
\begin{align}
  [x^0]\left\{
  \frac {at \left( at-ax+x \right) }{ \left( {x}^{2}{t}^{2}{a}^{2}+t{a}^{2}-at-ax+x \right) x}
  Q(0,\olx ) \right\}
  &= \frac{ \xi_+ }{at} \cdot \left. C_2 \right|_{b=1} \cdot Q(0,1/\xi_-).
\end{align}
Hence
\begin{align}
  Q(0,0)|_{b=1}
  & =  [x^0] H(x,Y(x);a,1) - \frac{\xi_+}{at}\cdot C_2|_{b=1}
\end{align}
and so is D-finite.

\section{Model 18} \label{sec soln 18}
We have been unable to solve this model for general $a,b$, however we have been able to solve it along the line $a=b$.
As remarked above, this model is very similar to that studied in \cite{tabbara2016}.
When $a=b$, the generating function is symmetric so that $Q(x,y) = Q(y,x)$, and we start by writing the functional equation and associated kernel:
\begin{align}
  \label{eqn 18original}
  abxy K(x,y) Q(x,y)
  &=
  \begin{multlined}[t]
  xy
  + ax\left(y(1-a) + at(x+1) \right)L(x) \\
  + ay\left(x(1-a) + at(y+1) \right)L(y)
  - (a-1)^2xy L(0)
\end{multlined}
\\
\text{and }  K(x,y) & = 1-t(x+\olx + y + \oly + x\oly + y\olx),
\end{align}
where we have written $Q(x,0) = L(x), Q(0,y)=L(y)$ and $Q(0,0) = L(0)$.

As noted in Table~\ref{tab 1718}, the kernel admits 6 symmetries which generate 5 new equations; we use those obtained by setting $(x,y)=(x,y), (\olx y,y)$ and $(\olx y,\olx)$.
An appropriate linear combination of these equations allows us to eliminate the unknown functions $L(y)$ and $L(\olx y)$.
We can then eliminate the bulk terms by  setting the kernel to zero by substituting $y = Y(t;x) \equiv Y(x)$ being the power series solution of $K(x,y)=0$:
\begin{align}
  Y(t;x) &\equiv Y(x) \nonumber \\
  &= \frac{x-t-tx^2-\sqrt{
\left( {x}^{4}-4{x}^{3}-6{x}^{2}-4x+1 \right) {t}^{2}-2x \left( {x}^{2}+1 \right) t+{x}^{2}
  }}{2t(1+x)} \\
  &= O(t).
\end{align}

This then gives
\begin{multline}\label{eqn 18fq1}
L(x) + {\frac { \left( Yat+ta-xa+x \right)  \left( xat+Yat-Yxa+Yx \right) }{\left( Yatx+xat-aY+Y \right)  \left( xat+Yat-a+1 \right) {x}^{2}}} L(\olx) \\
  =  H(x,Y(x);a)( 1- L(0)(a-1)^2 ),
\end{multline}
where $H$ is a somewhat complicated function that is rational in its arguments. To arrive at the solution we take the constant term of the above equation.

The constant term of $L(x)$ is simply $L(0)$. We can write the constant term of the right-hand side also in terms of $L(0)$:
\begin{align}
  ( 1- L(0)(a-1)^2 ) \cdot [x^0]\left\{ H(x,Y(x);a) \right\}.
\end{align}
The remaining term is more challenging.

As has been the case in many of the models we have discussed, one can show, with a little work, that the coefficient of $L(\olx)$ is actually a rational function of $t,x,a$:
\begin{align}
  {\frac { \left( a-1 \right)  \left( ta+a-1 \right) {x}^{2}- \left( a-1 \right)  \left( 2ta+1 \right) x+ta \left( 2ta+1 \right) }
 {x
 \left( ta \left( 2ta+1 \right) {x}^{2}- \left( a-1 \right)  \left(
2ta+1 \right) x+ \left( a-1 \right)  \left( ta+a-1 \right)  \right)
}}.
\end{align}
We are able to write it in a partial fraction form:
\begin{align}\label{eqn mod18 pf}
  C_0 + \frac{C_1}{x} + \frac{C_2}{1-x/\xi_+} + \frac{C_3}{1-\olx \xi_- },
\end{align}
where the $\xi_\pm$ are solutions of the denominator with respect to $x$:
\begin{align}
  \xi_+ &= {\frac {2t{a}^{2}-2ta+a-1+\sqrt {- \left( a-1 \right)
 \left( 2ta+1 \right)  \left( 4{t}^{2}{a}^{2}+2t{a}^{2}-2ta-a+
1 \right) }}{2ta \left( 2ta+1 \right) }} \\
& = O(t^{-1}); \nonumber \\
\xi_- &= {\frac {2t{a}^{2}-2ta+a-1-\sqrt {- \left( a-1 \right)
\left( 2ta+1 \right)  \left( 4{t}^{2}{a}^{2}+2t{a}^{2}-2ta-a+
1 \right) }}{2ta \left( 2ta+1 \right) }} \\
& = O(1). \nonumber
\end{align}
As was the case in Model 17, we have chosen this combination of $x,\olx, \xi_\pm$ so that the expression~\eqref{eqn mod18 pf} is a power series in $t$.
The coefficient $C_1$ is a simple rational function of $a,t$, while the other coefficients are algebraic. Note also that $C_0+C_3=0$.
Using this partial fraction form we can compute the constant term as
\begin{align}
  [x^0]\left\{ \left(   C_0 + \frac{C_1}{x} + \frac{C_2}{1-x/\xi_+} + \frac{C_3}{1-\olx \xi_- }
 \right)L(\olx) \right\}
 &= C_2 \cdot L(1/\xi_+).
\end{align}

Thankfully we can express $L(1/\xi_+)$ in terms of $L(0)$.
Return to equation~\eqref{eqn 18original}, and notice that we can eliminate $L(y)$ by setting $x=\frac{ta(1+y)}{a-1}$.
We can eliminate the kernel by setting $y=Y'(t;a)$ being the power series solution of
$K\left( \frac{ta(1+y)}{a-1}, y \right)=0$:
\begin{align}
  Y'(t;a) &=
  {a-1-\frac {2{t}^{2}{a}^{2}-\sqrt {- \left( a-1 \right)  \left( 2
  ta+1 \right)  \left( 4{t}^{2}{a}^{2}+2t{a}^{2}-2ta-a+1
   \right) }}{2ta \left( ta+a-1 \right) }}.
\end{align}
The combination
\begin{align}
  \frac{ta(1+Y')}{a-1} &= 1/\xi_+,
\end{align}
hence we are left with an equation linking $L(1/\xi+)$ with $L(0)$:
\begin{align}
  L(1/\xi_+) &= \frac{Y'(a-1)}{a(at+a-1)(atY'+at-aY'+Y')}\left(1-(a-1)^2 L(0) \right).
\end{align}
So putting together all the parts from the constant term we get the equation
\begin{multline}
  L(0) + C_2  \cdot \frac{Y'(a-1)}{a(at+a-1)(atY'+at-aY'+Y')}\left(1-(a-1)^2 L(0) \right) =\\
  ( 1- L(0)(a-1)^2 ) \cdot [x^0] \left\{ H(x,Y(x);a) \right\}.
\end{multline}
Isolating $L(0)$ gives
\begin{equation}
  L(0) = \begin{multlined}[t] \frac{1}{(a-1)^2}  \\ - \frac{1}{(a-1)^2 \left(1 -  \frac{C_2 \cdot Y'(a-1)}{a(at+a-1)(atY'+at-aY'+Y')} - (a-1)^2\cdot[x^0]\left\{ H(x,Y(x);a) \right\} \right)}.\end{multlined}
\end{equation}
Hence we have shown that $L(0) = Q(0,0) = G(t;a,a)$ can be written as (essentially) the reciprocal of a D-finite function. Therefore it is, at worst, D-algebraic.

\section{Model 19 --- Interacting Kreweras walks}
\label{sec soln 19}
We have so far been unable to solve this model for general $(a,b)$; however, we have solved it along the line $(a,a)$.
Our solution follows, essentially, Bousquet-M\'elou's \emph{algebraic kernel method} solution \cite{bousquet2005} but with additional complications.
The functional equation when $b=a$ and kernel are:
\begin{align}
  a^2 xy K \cdot Q(x,y) &= 1 + \left(ax-x-at\right) ayL(y)
  + \left(ay-y-at\right) axL(x)
  - \left(a-1 \right)^2 xy L(0)\\
  \text{and } K &= 1 - t(\olx + \oly + xy)\;,
\end{align}
where we have written $L(x) = Q(x,0) = Q(0,x)$.

Following the method in \cite{bousquet2005}, we use the kernel symmetries $(x,y)\mapsto(\olx \oly ,y), (x,\olx \oly )$ to generate two new equations.
A little linear algebra allows us to remove all of the $L(\olx \oly )$ terms and then we divide by the kernel:
\begin{multline}
 ( axyt-a+1) Q( x,y ) +
 ( at-ay+y ) \oly  Q(x,\olx \oly  ) -
  ( at-ax+x) \olx  Q(\olx \oly ,y) =
\\
\frac {a{x}^{2}{y}^{2}t-axy+axy-ayt+xy}{xy a^2 K}
-2 \frac{\left( at-ay +y \right)  \left( axyt - a + 1 \right) }{ay K} \cdot L \left( x \right)
\\
-\frac { \left( a-1 \right) ^{2}
 \left( a{x}^{2}{y}^{2}t-axy+axt-ayt+xy \right)  }{xy a^2 K} \cdot L \left( 0 \right).
\end{multline}

This is quite a bit messier than the $(a,b)=(1,1)$ case, but we can still proceed in the same manner; our next step is to take the constant term with respect to $y$ of both sides.
We do each side in turn.
\subsection{Constant term of the left-hand side}
Consider each of the 3 terms in turn.
\begin{itemize}
 \item The first term is easy to compute --- the constant term is
 \begin{align}
 [y^0] \left\{ (axyt-a+1) Q(x,y) \right\} &= (1-a) L(x).
\end{align}
\item The second term is a little messier, but again gives just
\begin{align}
[y^0] \left\{
(at-ay+y)\oly  Q( x, \olx \oly )
\right\}
&= (1-a)L(x).
\end{align}
\item The third is a little more complicated and we write its constant term in terms of the diagonal of $Q$:
\begin{align}
[y^0]\left\{
(ax-at-x)\olx  Q(\olx \oly ,y)
\right\}
&= (a-1 - at\olx  ) Q_d(\olx ),
\end{align}
where $Q_d(x) = [q^0] \left\{ Q(xq,\overline{q}) \right\}$ is the diagonal of $Q(x,y)$.
\end{itemize}
So the constant term of the left-hand side is
\begin{align}
  [y^0]\{\mathrm{LHS}\} &= 2(1-a)L(x) + (a-1-at\olx ) Q_d(\olx ).
\end{align}

\subsection{The constant term of the right-hand side}
The constant term of this side of the equation is more involved. We first note that the right-hand side can be written as
\begin{align}
  \mathrm{RHS} &= C_1 + \frac{C_2}{K},
\end{align}
where both $C_1,C_2$ are independent of $y$:
\begin{align}
  C_1 &= \frac{(a-1)^2}{a} L(0) + 2(1-a)L(x) - \frac{1}{a} ;\\
  C_2 &=
  {\frac { \left( a-1 \right) ^{2} \left( 2\,ta-x \right) L \left( 0
 \right) }{x{a}^{2}}}-2\,{\frac { \left( {a}^{2}{t}^{2}{x}^{2}+{a}^{2}
t-ta-xa+x \right) L \left( x \right) }{xa}}-{\frac {2\,ta-x}{x{a}^{2}}
}.
\end{align}
The constant term of $C_1$ is straightforward --- it is just $C_1$, but the constant term of $C_2/K$ requires some work.

It suffices to consider the constant term $[y^0]\left\{ \dfrac{1}{K} \right\}$. As for Model 17 we compute a partial fraction decomposition
\begin{align}
  \frac{1}{K} &= A_1 + \frac{A_2}{1-y/\gamma_+} + \frac{A_3}{1- \oly \gamma_- },
\end{align}
where $\gamma_\pm$ are $y$-solutions of the kernel:
\begin{align}
  \gamma_+ &= \frac{x-t+\sqrt{(x-t)^2-4t^2x^3}}{2tx^2} = \olx  t^{-1} + O(1); \\
  \gamma_- &= \frac{x-t-\sqrt{(x-t)^2-4t^2x^3}}{2tx^2} = t + O(t^2).
\end{align}
The constant term with respect to $y$ is $A_1+A_2+A_3$ and simplifies to
\begin{align}
  [y^0] \left\{ \frac{1}{K} \right \}
  &= A_1+A_2+ A_3 = \frac{1}{\sqrt{\Delta}},
  & \text{where } \Delta = (1-t\olx)^2-4xt^2.
\end{align}
So finally, the constant term of the right-hand side is just
\begin{align}
  [y^0] \{\mathrm{RHS} \} &= C_1 + \frac{C_2}{\sqrt{\Delta}}.
\end{align}

\subsection{Reassembly and another constant term}
Equating the constant terms of the two sides and cleaning up gives us
\begin{align}
   (a-1-at\olx ) Q_d(\olx ) - \frac{(a-1)^2}{a} L(0) + \frac{1}{a}
   &= \frac{C_2}{\sqrt{\Delta}}.
\end{align}
Notice that on the left-hand side of the above equation we only have non-positive powers of $x$.
We can further separate the powers of $x$ by careful factorisation of $\Delta$ into three series --- again following \cite{bousquet2005}
\begin{align}
  \Delta &= (1- t\olx )^2 - 4xt^3 = (1-x \delta_1)(1-\olx  \delta_2)(1-\olx  \delta_3),
  \intertext{where the roots $\delta_i$ are readily computed:}
  \delta_1 &= \frac{1}{4t^2}-2t-12t^4+\cdots \;;\\
  \delta_2 &= t + 2t^{5/2}+6t^4+ \cdots \;;\\
  \delta_3 &= t - 2t^{5/2}+6t^4 - \cdots \;.
\end{align}
Then write
\begin{align}
  \Delta_+ &= (1-x/\delta_1) = 1-4xt^2 -32xt^5 + \cdots \;,\\
  \Delta_- &= (1-\olx \delta_2)(1-\olx \delta_3) = 1-2\olx  t + \olx ^2 t^2 +\cdots \\
\text{and }  \Delta_0 &= 4t^2\delta_1 = 1-8t^3-48t^6+\cdots \;.
\end{align}

By multiplying through $\Delta_-$ we almost separate the equation into terms that contain strictly
negative powers of $x$ and terms that contain strictly positive powers of $x$:
\begin{multline}
\sqrt{\Delta_-}\left( \frac{(a-1)^2}{a} L(0) + \left( ta-xa+x \right) \olx  Q_d(\olx ) + \frac{1}{a} \right)
= \\
\frac{1}{\sqrt{\Delta_0 \Delta_+}} \left(
\frac { \left( a-1 \right) ^{2} \left( 2\,ta-x \right) L \left( 0 \right) }{x{a}^{2}}
-2 \frac { \left( {a}^{2}{t}^{2}{x}^{2}+{a}^{2}t-ta-xa+x \right) L \left( x \right) }{xa}
- \frac {2ta-x}{x{a}^{2}}
\right).
\end{multline}
Notice that the left-hand side now contains only powers of $x^0, x^{-1}, x^{-2},\dots$ while the right-hand side contains only powers of $x^{-1}, x^0, x^1, x^2, \dots$.
We can complete the separation of powers by moving the coefficient of $x^0$ from the left-hand side to the right, and moving the coefficient of $x^{-1}$ from the right-hand side to the left.

The constant term of the left-hand side is
\begin{align}
  [x^0]\{\mathrm{LHS}\} &= \frac{1}{a} + \frac{a-1}{a}L(0)
\end{align}
and the coefficient of $x^{-1}$ on the right-hand side is
\begin{align}
[x^{-1}]\{ \mathrm{RHS}\} &= \frac{2t((a-1)L(0) +1)}{\sqrt{\Delta_0}}.
\end{align}
Moving these terms to either side we can completely separate the equation into strictly negative powers of $x$ on the left-hand side, and non-negative powers of $x$ on the right-hand side.
This then gives two new equations (obtained by extracting non-negative powers of $x$ and strictly negative powers of $x$):
\begin{align}
  \sqrt{\Delta_-}\left( \frac{(a-1)^2}{a} L(0) + \left( ta-xa+x \right) \olx  Q_d(\olx ) + \frac{1}{a} \right)
  - \frac{2t( (a-1)L(0)-1)}{x\sqrt{\Delta_0}} &=0\;;
\end{align}
\begin{multline}\label{eqn krew nearly}
  \frac{1}{a} + \frac{a-1}{a}L(0) - \frac{2t( (a-1)L(0)-1)}{x\sqrt{\Delta_0}} \\
  -\frac{1}{\sqrt{\Delta_0 \Delta_+}} \left(
  \frac { \left( a-1 \right) ^{2} \left( 2\,ta-x \right) L \left( 0 \right) }{x{a}^{2}}
  -2 \frac { \left( {a}^{2}{t}^{2}{x}^{2}+{a}^{2}t-ta-xa+x \right) L \left( x \right) }{xa}
  - \frac {2ta-x}{x{a}^{2}}
  \right) \\ =0.
\end{multline}

The second of these is readily solved using the single-variable (ie ``classical'') kernel method.
Notice the coefficient of $L(x)$ is a polynomial in $x$. It has two $x$-solutions:
\begin{align}
  \xi_+ &=  \frac{a-1+\sqrt{(a-1)(1-a+4a^3t^3)}}{2a^2t^2} = \frac{a-1}{a^2} t^{-2} + O(t)\;;\\
  \xi_- &=  \frac{a-1-\sqrt{(a-1)(1-a+4a^3t^3)}}{2a^2t^2} = at + O(t^4).
\end{align}
Setting $x = \xi_-$ in equation~\eqref{eqn krew nearly} leaves an equation in only $L(0)$:
\begin{align}
  (1-a)L(0) &= \frac{a(2at-\xi_-)}{(2at-\xi)(1-a) + (\xi_-\sqrt{\Delta_0} - 2t)a\sqrt{\Delta_+(\xi_-)} },
\end{align}
where $\Delta_+(\xi_-)$ denotes substituting $x\mapsto \xi_-$ in $\Delta_+$.
Hence $L(0) = Q(0,0)$ is an algebraic function of $t$.
In fact, it satisfies a degree 6 polynomial.

\section{Discussion}\label{sec:discussion}
The results we have proved are summarised in Theorems~\ref{thm 1234} to~\ref{thm 22}.
We now summarise our unproven observations from Tables~\ref{tab 1to4} through~\ref{tab 2223} as three conjectures and one question.

While have found closed form expressions for $G(t;a,b)$ for several models as the ratio or reciprocal of D-finite functions --- suggesting a D-algebraic solution --- we are unable to prove that these are, in fact, not D-finite.
Additionally, we have not been able to guess D-finite solutions for these models based on analysis of long series.
This leads to the following conjecture.
\begin{conj}\label{conj 134171822}
  The generating functions $G(t;a,a)$ and $G(t;a,b)$ for Models 1, 3, 4, 17, 18 and 22 are not D-finite.
\end{conj}

We were able to find D-finite solutions for models possessing a symmetry across a vertical line, but only when $b=1$.
We have been unable to find or guess solutions for any other values of $a$ and $b$. On the basis of this we conjecture the following:
\begin{conj}\label{conj 5to16}
  The generating functions $G(t;a,b)$ for Models 5--16 are not D-finite except when $b=1$.
\end{conj}

The generating function $G(t;a,b)$ is equal in Models 19 and 20 and we have shown that it is algebraic when $a=b$. We have not solved the model at more general
values of $a$ and $b$, however we have managed to guess algebraic solutions at various integer values of $a$ and $b$ and based on this we conjecture:
\begin{conj}\label{conj kreweras}
  The generating functions $G(t;a,b)$ for Models 19 and 20, Kreweras and reverse Kreweras walks, are algebraic for all $a,b$.
\end{conj}

We have had far less success with Models 21 and 23 and have not been able to solve them at any non-trivial values of $a$ and $b$.
We were able to guess D-finite solutions to Model 21 at $a=1$ or $b=1$, but we have not been able to show that those solutions are transcendental.
We have not been able to guess algebraic solutions at those values. For Model 23 we were able to guess an algebraic solution when $b=1$ and a D-finite solution at $a=1$.
Similar to Model 21, we have not been able to prove transcendence of the D-finite solution at $a=1$, nor have we been able to guess an algebraic solution there.
Based on this we ask the question:
\begin{quest}
  Are the generating functions $G(t;a,1), G(t;1,b)$ for Model 21, double Kreweras walks, and $G(t;1,b)$ for Model 23, Gessel walks,  D-finite and not algebraic?

  If so, then the generating functions $G(t;a,b)$ for these models are not algebraic in contrast with $G(t;1,1)$.
\end{quest}

Some comments regarding solvability and symmetry are worthwhile.
Introducing different weights, $a$ and $b$, on the boundary breaks the diagonal symmetry of the quarter plane.
Models that are both horizontally and vertically symmetric (ie.~Models 1--4), remain soluble even though the analytic nature of the generating function changes --- see Conjecture~\ref{conj 134171822}.
The other problem that behaves in the same way is Model 17, which possesses a symmetry along the line $y=-x$. Breaking symmetry along the line $y=x$ does not seem to change its solubility.

On the other hand, Models 5--16 only have symmetry around a vertical line and we have only been able to solve these problems when there is no weight present on the vertical boundary.
As soon as the vertical boundary is weighted, the broken symmetry renders the problem more difficult and we have been unable to either solve or guess the generating function --- see Conjecture~\ref{conj 5to16}.

We have been unable to solve Models 18, 19, 20 and 21 when $a \neq b$.
These four models possess symmetry along the line $y=x$ (but not horizontal or vertical symmetry), and unsurprisingly our solution method, which relies on this symmetry, fails in these cases.
In spite of this, we have strong numerical evidence that $G(t;a,b)$ is algebraic for Models 19 and 20 --- see Conjecture~\ref{conj kreweras}.

We finish by making a few comments about the impediments to using the variations of the kernel method to the problems above.
Arguably, these difficulties fall into three categories.
First, and perhaps most prosaic, is that when we introduce boundary weights, the coefficients of equations become more complicated.
While, in principle, this does not prevent a solution, in practice it does.
For example, we believe that Model 21 might be soluble via a similar approach to that used for Models 19 and 20, however the required manipulations quickly become byzantine.

Perhaps more substantive is that for a great number of the models we have not solved --- including all of 5--16 --- a key coefficient is no longer rational.
More specifically, if we attempt to replicate the approach in Section~\ref{sec soln 134}, the coefficient of $Q(\olx,0)$ in the equivalent of equation~\eqref{eqn:nsew_fe_after_halforbit_cancelkernel} is algebraic and not rational.
Because of this, we are not able to express the required constant term as a simple substitution. Indeed, it is not clear that a simple closed form for the required constant term exists.

Finally, we have attempted to use the so-called full orbit sum kernel method, in which one sums over all the symmetries of the kernel in order eliminate boundary terms --- see Proposition~8 from \cite{bousquet-melou_walks_2010}.
Unfortunately, it does not appear to be possible to eliminate sufficient boundary terms.
For example, in Model 5, a sum over the 4 kernel symmetries allows one to remove most, but not all, of the boundary terms.
For general $a,b$, one is left with an equation of the form
\begin{align}
  \Big[ \text{ bulk terms }\Big] K &= \Big[ \text{ coefficient } \Big] Q(0,0) + \Big[ \text{ coefficient } \Big] Q(x,0)
\end{align}
and it does not appear that any improvement is possible. It may be possible to follow the methods used to solve Kreweras walks (see \cite{bousquet2005} and Section~\ref{sec soln 19}) to further process this equation towards a solution.
This is the subject of ongoing work \cite{xu2018}.

It is certainly worth investigating whether other methods, such as recent differential Galois theory \cite{dreyfus2018} and the Tutte invariant method \cite{bernadicounting} can be adapted to handle these interacting boundary problems.

\section*{Acknowledgements}

Financial support from the Australian Research
Council via its Discovery schemes (DE170100186 and DP160103562)
is gratefully acknowledged by N.~R.~Beaton and A.~L.~Owczarek respectively.
A.~Rechnitzer acknowledges support from NSERC Canada via a Discovery Project Grant.
N.~R.~Beaton and A.~Rechnitzer also recieved support from the PIMS Collaborative
Research Group in Applied Combinatorics.


\end{document}